\documentclass[a4paper,11pt]{article}
\usepackage[english]{babel}
\usepackage{graphicx,color,epsfig}
\usepackage{amsmath,amscd,amsfonts,amssymb,amsthm,euscript}

\newtheorem{trm}{Theorem}[section]
\newtheorem{prop}{Proposition}[section]

\newtheorem{oss}{Remark}[section]

\def \R{\mbox{${\mathbb R}$}}
\def \h{\mbox{$\mathbb H$}}

\def \s{\mbox{${\mathbb S}$}}

\begin{document}
\title{On the three-dimensional homogeneous \\
SO(2) - isotropic Riemannian manifolds}
\author{ P. Piu and M. M. Profir }

\date{}

\maketitle
\begin{abstract}
In this paper we consider some properties of the three-dimensional homogeneous  SO(2)-isotropic
Riemannian manifolds. In particular, we determine the geodesics, the
totally geodesic surfaces, the totally umbilical surfaces and the
geodesics of the rotational surfaces.
\end{abstract}
\section{Introduction}
 We consider a two-parameter family of
three-dimensional Riemannian mani\-folds $(M,ds^{2}_{\ell,m})$,
where the metrics have the expression
\begin{equation}\label{eq:Cartan-Vranceanu}
  ds^{2}_{\ell,m} =\frac{dx^{2} + dy^{2}}{[1 + m(x^{2} + y^{2})]^{2}} +
\left(dz + \frac{\ell}{2} \frac{ydx - xdy}{[1 + m(x^{2} +
y^{2})]}\right)^{2},
\end{equation}
with $\ell,m \in {\R}$. The underlying differentiable manifolds
$M$ are $\R ^3$ if $m\geq 0$ and $M=\{(x,y,z)\in\R^3\;:\; x^{2} +
y^{2} < - \frac{1}{m}\}$ otherwise.
 \noindent These metrics can be found in the classification of
$3$-dimensional homogeneous metrics given by L. Bianchi in 1897
(see \cite{[Bi]}); later, they appeared in form
(\ref{eq:Cartan-Vranceanu}) in \'{E}. Cartan (\cite{[Ca]}) and G.
Vranceanu (see \cite{[Vr]}). For these reasons we call them the
Cartan-Vranceanu metrics (\textbf{C-V metrics}). Their geometric
interest lies in the following: {\it the family of metrics
(\ref{eq:Cartan-Vranceanu}) includes all $3$-dimensional
homogeneous metrics whose group of isometries has dimension $4$ or
$6$, except for those of constant negative sectional curvature}.
We recall the spaces that correspond to the different values of $\ell$
and $m$.
\begin{enumerate}
\item[\bf $\bullet$] If $\ell=0$, then $M$ is the product of a surface
$S$  with constant Gaussian curvature $4m$ and the real line $\R$.
\item[\bf $\bullet$]If  $ 4m - \ell^2 = 0 $, then
$M$ has nonnegative constant sectional curvature.
\item[\bf $\bullet$] If $\ell \neq 0$ and $m>0$, $M$ is
locally  $SU(2)$.
\item[\bf $\bullet$]Similarly, if  $\ell \neq 0$ and $m < 0$, $M$ is
locally $\widetilde{SL}(2,\R)$, while if
\item[\bf $\bullet$] $ m = 0$ and $\ell \neq 0$
we get a left invariant metric on the Heisenberg Lie group
$\h_{3}$.
\end{enumerate}

The isometry group of these spaces
 has a subgroup isomorphic
to the group $SO(2)$, so there are rotational surfaces around $z$-axis.
R. Caddeo, P. Piu, A. Ratto and P. Tomter (see \cite{[CPR$_1$]},
\cite{[CPR$_2$]}, \cite{[To]}) have studied rotational surfaces in
$\h_{3}$ with constant (mean or Gauss) curvature, while the  CMC and CGC invariant surfaces of  $\h_3$ and of
the product space $\h^2 \times \R$ have been studied by C. Figueroa, F. Mercuri, R. Pedrosa,  S. Montaldo and I. Onnis (see \cite{[FMP]},\cite{[MO1]}, \cite{[MO]}). In this paper we obtain the
Lie algebra of the Killing vector fields and thus the group of
isometries for the C-V metrics. We explicitly determine  the equations of the
geodesics by using the Killing vector fields and  obtain
the equations of the surfaces which contain the geodesics. After
having determined the totally geodesic surfaces isometrically
immersed in the C-V spaces, we study the totally umbilical
surfaces of these spaces, proving that the only totally
umbilical surfaces are totally geodesic. We find the geodesics
for the $SO(2)$-invariant surfaces of the Cartan-Vranceanu spaces,
deduce the conditions that meridians and parallels must satisfy in
order to be geodesics and show the analogies with the Euclidean
case.

\section{Geodesics on C-V spaces}
It is well known that a curve  $\gamma:I\rightarrow M$ on a
Riemannian manifold $(M, g,\nabla)$ with the Levi-Civita
connection $\nabla$ is a {\bf geodesic} if its velocity
vector field is constant (parallel),
\begin{equation}
\nabla_{\dot{\gamma}}\dot{\gamma}=0.
\end{equation}
We also remember an important theorem of Levi - Civita:
\begin{trm}\label{Levi-Civita}
If $X$ is a  Killing vector field for the Riemannian manifold
$(M,g)$ then the equation of the geodesics $\gamma(t)$ admits the
prime integral
\[
g(\dot{\gamma}, X ) = const.
\]
\end{trm}
\begin{proof}
The derivative with respect to $t$ of the scalar product
$g(\dot{\gamma},X)=\varphi(t)$
 gives
$$
\frac{d}{dt}g(\dot{\gamma},X)=g(\nabla_{\dot{\gamma}}\dot{\gamma},X)+g(\dot{\gamma},\nabla_{\dot{\gamma}}X)
$$
and this is zero, because $\gamma$ is a geodesic and
  $X$ a Killing vector field. Thus we have
  $\varphi=const$.
\end{proof}

We want to obtain the geodesics for the simply connected
homogeneous $SO(2)$-isotropic $3$-dimensional Riemannian manifolds
with isometry group of dimension $4$, endowed with the C-V
metrics.\\
 The Cartan-Vranceanu metric \eqref{eq:Cartan-Vranceanu} can be written as:
 \begin{equation}\label{1.2}
 ds^{2}_{\ell,m} =\sum_{i=1}^3 \omega^i \otimes \omega^i,
\end{equation}
 where, putting $D = 1+m(x^2 + y^2)$,
\begin{equation}\label{forme}
\omega^1 = \frac{dx}{D} \hspace{1 cm}
\omega^2 = \frac{dy}{D} \hspace{1 cm} \omega^3 = dz + \frac{\ell}{2} \frac{ydx - xdy}{D}.
\end{equation}
The orthonormal basis of vector fields dual to the $1$-forms \eqref{forme} is

\begin{equation}\label{baseorto}
E_1 = D \frac{\partial}{\partial x} - \frac{\ell}{2} y \frac{\partial}{\partial z} \qquad
E_2 = D \frac{\partial}{\partial y} +\frac{\ell}{2} x \frac{\partial}{\partial z}  \qquad
E_3 = \frac{\partial}{\partial z}.
\end{equation}
 The Killing vector fields of the metric (\ref{eq:Cartan-Vranceanu})  are the vector fields $X=\xi^{i} E_{i}$ such that the Lie derivative with respect to $X$ of the metric is zero
\[
L_{X}(ds^{2}_{\ell,m})  = 0.
\]
A basis for the Lie algebra of Killing vector fields has been
computed (see  \cite{[P]}) and we found that it is given by
\begin{align}
X &= \frac{2 m xy}{D}E_1 + \left(1 - \frac{2 m x^2}{D}\right)E_2 - \frac{\ell x}{D}E_3\notag\\
Y&= \left(1 - \frac{2 m y^2}{D}\right)E_1 + \frac{2 m xy}{D}E_2 +  \frac{\ell y}{D}E_3\notag\\
Z&= E_3 \notag\\
R &=  - \frac{y}{D}E_1 + \frac{x}{D}E_2 - \frac{\ell(x^2 + y^2)}{2D}E_3 \notag.
\end{align}
\vspace{.5 cm}
 Let $\gamma: I \rightarrow M$ be a geodesic on the
manifold $(M,ds^{2}_{\ell,m})$. The tangent vector field
$\dot{\gamma}$ with respect to the orthonormal basis
\eqref{baseorto} is
\[
\dot{\gamma} = \frac{\dot{x}}{D} E_1 + \frac{\dot{y}}{D}E_2  +
\left[\dot{z} - \frac{\ell}{2} \frac{x \dot{y} -  y
\dot{x}}{D}\right] E_3.
\]
 \vspace{.5 cm}
According to Theorem \ref{Levi-Civita} we can write four
prime integrals
\[
\begin{cases} \displaystyle\frac{ 2 m xy\dot{x}}{D^2} +
\displaystyle\frac{(1 + m ( y^2 - x^2)) \dot{y}}
{D^2} - a_3  \displaystyle\frac{\ell x}{D}  =   a_1\\
\\
 \displaystyle
 \frac{( 1 + m (x^2 - y^2))\dot{x}}{D^2} +
 \frac{ 2 m xy\dot{y}}{D^2} + a_3 \frac{\ell y}{D} =  a_2 \\
 \\
\dot{z} -  \displaystyle
\frac{\ell}{2} \frac{x \dot{y} -  y \dot{x}}{D} = a_3 \\
\\
 \displaystyle
 \frac{\dot{y} x}{D^2} -
  \displaystyle
 \frac{\dot{x} y}{D^2}- a_3
\frac{\ell(x^2 + y^2)}{2 D}  = a_4
\end{cases}
\qquad a_i \in \R.
\]

\begin{oss}
As here we are considering homogeneous riemannian spaces, in order to obtain all the
geodesics starting at a point $p$ it is sufficient to translate the
geodesics starting at the origin of the coordinate system by using the isometry that takes $0$ to $p$.
\end{oss}
Considering thus the geodesics starting at origin such that
$\dot{\gamma}(0)=(u,v,w)$, the constants $a_{i}$ take the values
\[
a_{1}=v \qquad  a_{2}=u \qquad a_{3}=w \qquad a_{4}=0
\]
and the prime integrals become
\begin{equation}\label{integrali-primi}
\begin{cases}
\displaystyle\frac{ 2 m xy\dot{x}}{D^2} +
\displaystyle\frac{(1 + m ( y^2 - x^2)) \dot{y}}
{D^2} - \displaystyle\frac{\ell x w}{D}  = v\\
\\
\displaystyle \frac{( 1 + m (x^2 - y^2))\dot{x}}{D^2} +
 \displaystyle\frac{ 2 m xy\dot{y}}{D^2}  + \displaystyle \frac{\ell yw}{D} = u \\
 \\
\dot{z} -  \displaystyle\frac{\ell}{2}\displaystyle \frac{x \dot{y} -  y \dot{x}}{D} = w \\
\\
 \displaystyle \frac{\dot{y} x-\dot{x} y}{D^2}
   -
\displaystyle\frac{\ell(x^2 + y^2)w}{2 D}  = 0.
\end{cases}
\end{equation}

 We observe that the Killing vector field $R$ that
generates the rotations around the $z$-axis, may be written as a
combination of the other Killing vector fields:
\[
R = \frac{- y}{ m(x^2 + y^2) - 1} X + \frac{x}{m(x^2 + y^2) -1} Y
- \frac{\ell(x^2 + y^2)}{2(m(x^2 + y^2) -1 )} Z .
\]
Theorem \ref{Levi-Civita}, applied to $R$, gives the
following prime integral
 \[
\frac{- y}{m(x^2 + y^2)-1} a_1 + \frac{x}{m(x^2 + y^2)-1} a_2 -
\frac{\ell(x^2 + y^2)}{2(m(x^2 + y^2)-1)} a_3 = 0,
\]
equivalent to the equation
  \begin{equation}\label{cilindro}
\ell(x^2 + y^2) w - 2 x v + 2y u = 0.
\end{equation}
Then we have
\begin{prop}\label{prop-cilindro}
 For the C-V metrics, the geodesics $\gamma(t)$ starting at the origin and such  that $\dot{\gamma}(0) = (u,v,w)$ can be defined as the intersection
 of two surfaces:
\begin{itemize}
\item  a circular cylinder with the generating line (ruling)  parallel to the $z$-axis or a plane parallel to the $z$-axis;
\item a surface of rotation around the $z$-axis.
\end{itemize}
\end{prop}
\begin{proof}
According to equation \eqref{cilindro}, for $\ell \neq 0$ and
$w \neq 0$ we find that the geodesics of the Heisenberg group $\h_3$,
of the Lie group $SU(2)$ and those of the universal
covering of $SL(2,\R)$ are contained in a {\it cylinder} at origin
with generating lines
parallel to the $z$-axis; for $\ell \neq 0$ and $w =0$ the geodesics  are contained in the plane $vx - u y= 0$ parallel to the $z$-axis.\\
For $\ell = 0$ we obtain that the geodesics of the product spaces $\s^2
\times \R$ and $\h^2 \times \R$
 are contained in the plane $vx - u y= 0$.\\
The rotational surface is obtained by rotating a geodesic
around the $z$-axis.
\end{proof}
We shall describe briefly the method of finding the equations of
geodesics taking into consideration the case $m\neq
0$ and $ \ell \neq 0$, when the family of metrics
\eqref{eq:Cartan-Vranceanu} gives a metric of the spaces $SU(2)$
and $\widetilde{SL}(2,\R)$. It is convenient to write the integrals in
\eqref{integrali-primi} in cylindrical coordinates.
Considering the geodesic starting at origin and tangent  at
the vector $(u,v,w)$, the prime integrals of the equation of the
geodesics and the unit norm of the vector $\dot{\gamma}$ give
the system

\begin{equation}\label{integraliprimi}
\begin{cases}
\displaystyle\frac{\dot{\rho} \cos \theta}{1 + m\rho^{2}} - \frac{ \dot{\theta} \sin \theta}{(1+m\rho^{2})^{2}}(\rho - m \rho^3) +\frac{\ell \rho \sin \theta}{1+m\rho^{2}}w=u\\
\mbox{}\\
\displaystyle\frac{\dot{\rho} \sin \theta}{1 + m\rho^{2}} + \frac{ \dot{\theta} \cos \theta}{(1+m\rho^{2})^{2}}(\rho - m \rho^3) -\frac{\ell \rho \cos \theta}{1+m\rho^{2}}w=v\\
\mbox{}\\
\displaystyle\dot{z}-\frac{l}{2}\frac{\rho^{2}\dot{\theta}}{1+m\rho^{2}}=w\\
\mbox{}\\
\displaystyle\frac{\rho^{2}\dot{\theta}}{(1+m\rho^{2})^{2}}-\frac{l}{2}\frac{\rho^{2}w}{1+m\rho^{2}}=0\\
\mbox{}\\
\displaystyle\frac{\dot{\rho}^{2}+\rho^{2}\dot{\theta}^{2}}{(1+m\rho^{2})^{2}}=u^{2}+v^{2}
\end{cases}
\end{equation}
For $w = 0$, the system
\eqref{integraliprimi} becomes
\[
\begin{cases}
\dot{\rho} \cos \theta
= u ({1 + m\rho^{2}})
\\
\dot{\rho} \sin \theta
= v ({1 + m\rho^{2}})\\
\displaystyle\dot{z}
= 0\\
\dot{\theta}
= 0\\
\displaystyle\frac{\dot{\rho}^{2}
}{(1+m\rho^{2})^{2}}
= u^{2}+v^{2}
\end{cases}
 \qquad \Longrightarrow \qquad
\begin{cases}
\dot{\rho} \cos \theta
 = u ({1 + m\rho^{2}})\\
 \dot{\rho} \sin \theta= v ({1 + m\rho^{2}})\\
z  = 0 \\
\theta = a\\
\\
\displaystyle\frac{\dot{\rho}}{(1+m\rho^{2})} = \pm
\sqrt{u^{2}+v^{2}}.
\end{cases}
\]
The immediate integration of the last equation gives the equations
of the geodesics for the spaces $SU(2)$ and
$\widetilde{SL}(2,\R)$, respectively :
\[
\begin{cases}
x =\displaystyle \frac{u}{\sqrt{u^2 + v^2}} \frac{\tan (\sqrt{m(u^2 + v^2)} t)}{\sqrt{m(u^2 + v^2)}} \\
\\
y = \displaystyle \frac{v}{\sqrt{u^2 + v^2}}\frac{\tan (\sqrt{m(u^2 + v^2)} t)}{\sqrt{m(u^2 + v^2)}}  \\
\\
z = 0
\end{cases}
\quad
\begin{cases}
x =\displaystyle  \frac{u}{\sqrt{u^2 + v^2}} \frac{\tanh (\sqrt{-m(u^2 + v^2)} t)}{\sqrt{-m(u^2 + v^2)}} \\
\\
y =\displaystyle \frac{v}{\sqrt{u^2 + v^2}}\frac{\tanh (\sqrt{- m(u^2 + v^2)} t)}{\sqrt{- m(u^2 + v^2)}}  \\
\\
z = 0
\end{cases}
\]

 If $w\neq 0$, from the last  two equations of
system \eqref{integraliprimi} we have
$$
d{\theta}=\frac{\ell w}{2}(1+m\rho^{2}) dt, \qquad \quad \frac{d\rho}{(1+m\rho^{2})\sqrt{(u^{2}+v^{2})-\frac{\ell^{2}}{4}w^{2}\rho^{2}}}=\pm
dt.
$$

Now we put $\frac{\ell^{2}}{4}w^{2}=a^{2},\quad u^{2}+v^{2}=b^{2}$ and $a^2 + b^2m \neq 0$.
Then by integrating we obtain
\begin{equation*}\label{geo-SU2}
\rho^2 = \frac{b^2 \tan At}{A^2 + a^2 \tan At}\qquad \quad
\theta = \arctan \frac{\ell w \tan At}{A},
\end{equation*}
where $A=\sqrt{\ell^2  w^2 + 4 m (u^2 + v^2) }$.
We shall give a list of all geodesics obtained together with
their graphical representation.
\begin{itemize}
\item For $\ell \neq 0$  and $\ell^2  w^2 + 4 m (u^2 + v^2) > 0$,  we have the following equations:
\[
\begin{cases}
x = \displaystyle\frac{2 \tan(\frac{At}{2})}{\sqrt{A^{2}+ \ell^2 w^2 \tan^{2}(\frac{At}{2})}} \left(u \cos T - v \sin T\right)\\
\mbox{}\\
y = \displaystyle\frac{2 \tan(\frac{At}{2})}{\sqrt{A^{2}+ \ell^2 w^2 \tan^{2}(\frac{At}{2})}} \left(v \cos T + u \sin T\right)\\
\\
z=wt-\displaystyle\frac{\ell^{2}w}{4m}t - \frac{\ell w}{2 m} T,\\
\end{cases}
\]
with $T=\arctan\frac{\ell
w\tan(\frac{A t}{2})}{A} $.

If $m >0$ and $4 m \neq \ell^2$ these equations determine the geodesics of
$SU(2)$,  while if $4 m = \ell^2$ we have the geodesics of the
sfere $\s^3$.
If $m < 0$ these equations determine the geodesics of
$\widetilde{SL}(2,\R)$.

\begin{figure}[h!]
\begin{center}
    \begin{minipage}[b]{2 in}
    \epsfxsize=0.65 in
    \centerline{
    \leavevmode
\epsffile{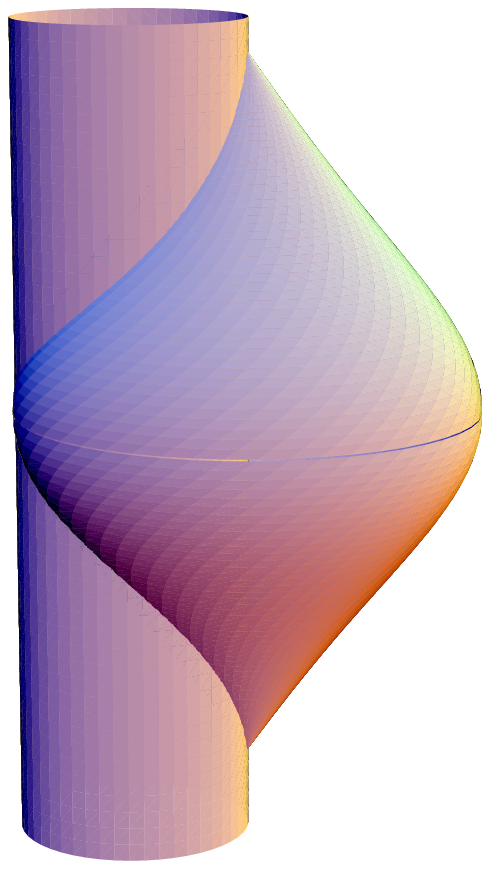}}
       \centerline{
       \protect{Geodesics of $SU(2)$} \label{fig5}}
    \end{minipage}
    \hspace{10mm}
    \begin{minipage}[b]{2.0in}
    \epsfxsize=0.5 in
    \centerline{
    \leavevmode
    \epsffile{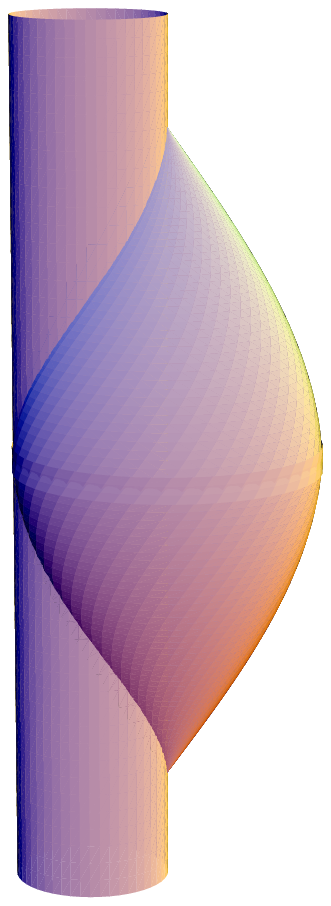}}
       \centerline{ \protect{Geodesics of $\s^3$}}
     \end{minipage}\label{fig4}
    \hspace{10mm}
\end{center}
\end{figure}

\item For $\ell \neq 0$  and $\ell^2  w^2 + 4 m (u^2 + v^2) < 0$,($m<0$), the parametric equations of the geodesics starting at the  origin of
$\widetilde{SL}(2,\R)$ in the case $w \neq 0$ are:
\[
\begin{cases}
x = \displaystyle\frac{2 \tanh(\frac{C t}{2})}{\sqrt{C^{2}+ \ell^2 w^2 \tanh^{2}(\frac{C t}{2})}} \left(u \cos T' - v \sin T'\right)\\
y = \displaystyle\frac{2 \tanh(\frac{C t}{2})}{\sqrt{C^{2}+ \ell^2 w^2 \tanh^{2}(\frac{C t}{2})}} \left(v \cos T' + u \sin T'\right)\\
z=w t-\displaystyle\frac{\ell^{2}w}{4m}t - \frac{\ell w}{2 m} T',\\
\end{cases}
\]
where $C=\sqrt{- \ell^2  w^2 - 4 m (u^2 + v^2) }$ and $T'
=\arctan\displaystyle\frac{\ell w\tanh(\frac{C t}{2})}{C} $

For $\ell \neq 0$  and $\ell^2  w^2 + 4 m (u^2 + v^2) = 0$,($m<0$), the parametric equations of the geodesics starting at the origin of
$\widetilde{SL}(2,\R)$ are, for $w \neq 0$,
\[
\begin{cases}
x = \displaystyle \frac{2 t}{\sqrt{ 4 + \ell^2 w^2 t^2}}\left( u \cos T - v \sin T\right) \\
y = \displaystyle \frac{2 t}{\sqrt{ 4 + \ell^2 w^2 t^2}}\left( v \cos T + u  \sin T\right) \\
z = w t - \displaystyle\frac{\ell^2 w t}{4 m} +
\displaystyle\frac{\ell T}{2m}
\end{cases}
\qquad T
=\arctan\displaystyle\frac{\ell w t}{2} .
\]
\item If $m=0$ and $\ell\neq 0$, the parametric equations of the geodesics arising from the
origin of the Heisenberg group $\h_3$  in the cases $w \neq 0$ and $w = 0$, are, respectively,
\[
\begin{cases}
x(t) = \displaystyle \frac{v}{\ell w}\cos(\ell w t) + \frac{u}{\ell w}\sin(\ell w t) - \frac{v}{\ell w} \\
y(t) = \displaystyle\frac{v}{\ell w}\sin(\ell w t) - \frac{u}{\ell w} \cos(\ell w t) + \frac{u}{\ell w} \\
z(t) = w t + \displaystyle\frac{u^2 + v^2}{ 2 w} t  - \frac{u^2 +
v^2}{ 2 w}\sin (\ell wt)
\end{cases}
\qquad
\begin{cases}
x = u t \\
y =v t \\
z = 0
\end{cases}
\]
\begin{figure}[h!]
\begin{center}
\begin{minipage}[b]{2 in}
    \epsfxsize=0.15 in
    \centerline{
    \leavevmode
    \epsffile{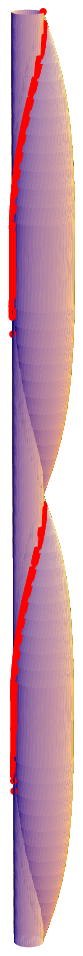}}
      \centerline{\protect{Geodesics of $\widetilde{SL}(2,\R)$\label{fig6}}}
\end{minipage}
\hspace{10mm}
\begin{minipage}[b]{2.0in}
    \epsfxsize=0.35 in
    \centerline{
    \leavevmode
    \epsffile{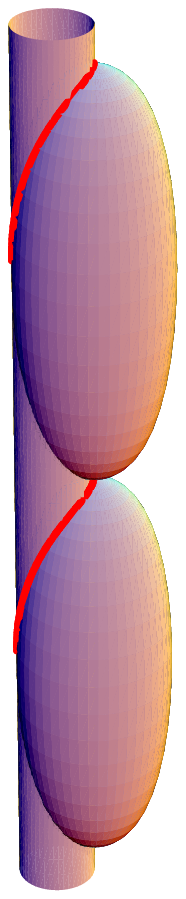}}
       \centerline{ \protect{Geodesics of $\h_3$ }\label{fig3}}
        \end{minipage}\label{fig4}
\end{center}
\end{figure}

    \item If  $m>0$, $\ell = 0$, we have the following cartesian equations of the geodesics starting at the origin of
    $\s^{2}\times\R$,
     in the cases $w \neq 0$ and  $w = 0$:
\[
\begin{cases}
v x - u y = 0 \\
x^2 + y^2 = \displaystyle \frac{1}{m}
\tan^2(\frac{\sqrt{m(u^{2}+v^{2})}}{w} z)
\end{cases}
\qquad
\begin{cases}
v x - u y = 0 \\
z = 0
\end{cases}.
\]

\item If  $m<0$, $\ell = 0$, we have the following cartesian equations of the geodesics starting at the origin of $\h^2 \times \R$,
      respectively
     in the case $w \neq 0$ and  $w = 0$:
\[
\begin{cases}
v x - u y = 0 \\
x^2 + y^2 = - \displaystyle \frac{1}{m} \tanh^2(\frac{\sqrt{-
m(u^{2}+v^{2})}}{w} z)
\end{cases}
\qquad
\begin{cases}
v x - u y = 0 \\
z = 0
\end{cases}
\]

\begin{figure}[h!]
\begin{center}
\begin{minipage}[b]{2.0in}
    \epsfxsize=0.65 in
    \centerline{
    \leavevmode
    \epsffile{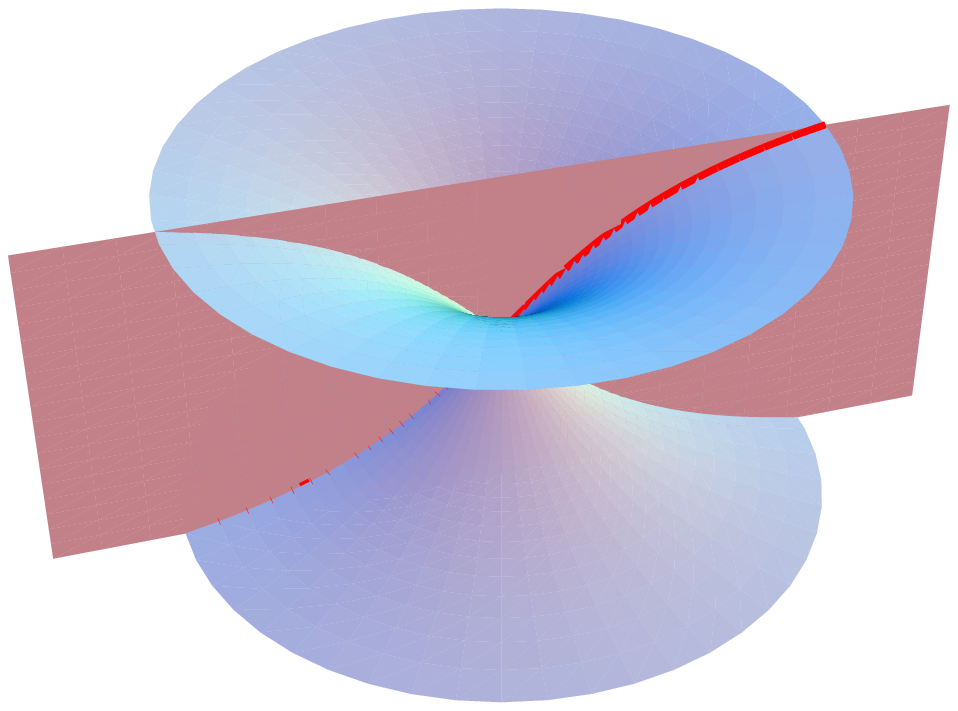}}
        \centerline{\protect{Geodesics of  $\s^{2}\times\R$}\label{fig1}}
        \end{minipage}
    \hspace{20mm}
    \begin{minipage}[b]{2 in}
    \epsfxsize=0.45 in
    \centerline{
    \leavevmode
    \epsffile{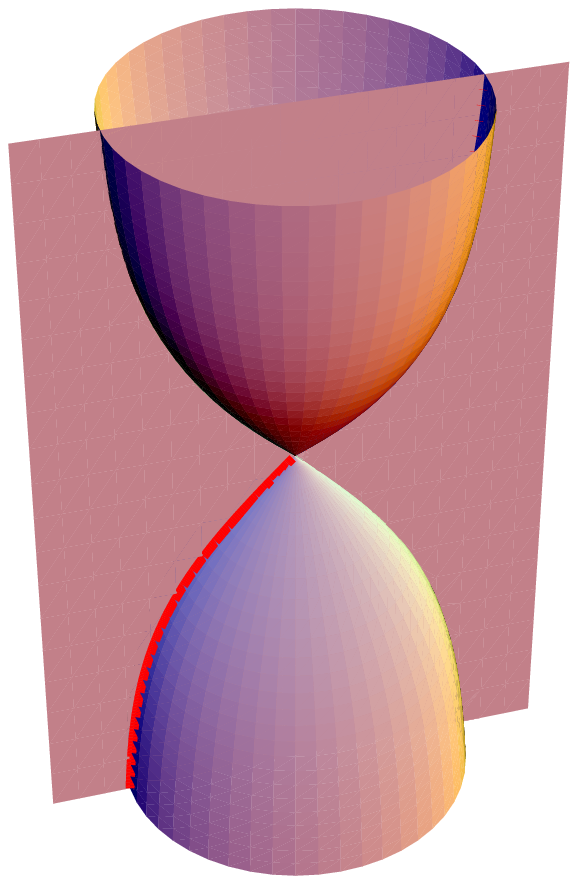}}
      \centerline{\protect{Geodesics of $\h^{2}\times\R$}\label{fig2}}

\end{minipage}
\end{center}
\end{figure}
\end{itemize}

\section{Totally geodesic submanifolds}
 By definition, a submanifold $V \subset M$   is \textbf{totally geodesic} if
  each of its geodesics is also a geodesic of $M$. We recall the
  following
\begin{trm}
A submanifold  $(V, g)$ of a Riemannian manifold $(M,
\overline{g})$ is totally geodesic if and only if its second
fundamental form is identically zero.
\end{trm}
In \cite{[H]}, Hadamard took into consideration the problem of
determining, locally,  the  Riemannian manifolds  $(M,\overline{g})$,
of dimension $3$, endowed with a foliation $\frak F$ so that each
geodesic of $M$ tangent to a leaf at a point is completely
contained in the leaf. Then the leaves of the foliation $\frak F$ are
 totally geodesic submanifolds and that is the reason why such a
foliation is called  {\bf totally geodesic}.

Cesare Rimini in~\cite{[Rm]} studied the problem proposed by
Hadamard. He looked for the three-dimensional manifolds with a
complete group of isometries of dimension $4$ that admit totally
geodesic foliations. He gave some important properties of these
manifolds, found the conditions so that such manifolds admit
totally geodesic foliations and, finally, determined
 these foliations. \\
The results obtained by Rimini are resumed in the following theorems.
\begin{trm}\label{trm:3}{\em (Rimini)}
 A Riemannian manifold $(M,\overline{g})$
admits a foliation $\frak F$ of totally geodesic hypersurfaces of
equation $x^{n}=const.$ if and only if they are isometric and the
isometries are determined by their orthogonal trajectories.
\end{trm}
\begin{trm}\label{trm:4}{\em (Ricci - Rimini)}
If a space $(M,g)$, $dim(M)=3$, contains a family of totally
geodesic surfaces, then this consists on surfaces orthogonal to a
principal Ricci direction $\xi$.\\
The principal Ricci curvature of the space, relative to $\xi$,
calculated at $p$, is equal to the Gauss curvature of the surface
of the family mentioned above passing through $p$, and thus it is
constant along every principal curve.
\end{trm}
\begin{trm}\label{trm:5}{\em (Rimini)}
In a Riemannian manifold of dimension $3$, with isometry group  of
dimension $4$, there are not totally geodesic surfaces, except for
the product spaces. In these spaces there are two types of totally
geodesic surfaces:
\begin{itemize}
\item[-] the surfaces orthogonal to the systatic geodesics, that form a family of totally geodesic surfaces $\frak{F}$
         with non zero constant gaussian curvature;
\item[-] a {\it cylinder} having as generating lines the
systatic lines and as gene \-rating curve a geodesic of one of the
totally geodesic surfaces of the foliation $\frak F$. Each of these
totally geodesic surfaces has the Gaussian curvature equal to zero.
\end{itemize}
\end{trm}
We want to determine the totally geodesic surfaces
isometrically immersed in the Cartan-Vranceanu manifolds.\\
According to Theorem \ref{trm:4} we have that a foliation
$\mathfrak F$ is  totally geodesic if at each point the orthogonal
trajectory to $\frak F$ is an isometry generated by a principal
Ricci direction. The principal Ricci directions of $(M,ds_{\ell
m}^{2})$ are determined by the vector fields $E_1, E_2, E_3$ in
\eqref{baseorto} and the only isometry  generated by a principal
Ricci direction  is that generated by $E_3$.  Hence, supposing
that the foliation $\quad \mathfrak F \quad$ is totally geodesic,
the leaves must be orthogonal to the principal Ricci direction
 $E_3 = Z =\frac{\partial}{\partial z}$ and we have that
\[
\mathfrak F  = \mbox{Ker}(\omega),
\]
where
\[
 \omega = dz  + \frac{\ell}{2} \frac{y dx - x  dy }{[1 + m(x^{2}
+ y^{2}]}.
 \]
But
\[
\omega \wedge d\omega = \ell dx\wedge dy \wedge dz,
\]
and therefore the form $\omega$ will be integrable if and only if $\ell = 0$. So
we have proved the following
\begin{trm} In a
Cartan-Vranceanu space $(M,ds_{\ell m}^{2})$ there are not totally
geodesic surfaces, with the exception of  the product spaces
$$\s^2(c) \times \R \qquad \h^2(-c) \times \R.$$
In such spaces there are two types of  totally geodesic surfaces:
\begin{itemize}
\item  the surfaces $\s^2(c) \times \{a\}$ and $\h^2(-c) \times \{a\}$;
\item a {\it cylinder} having as generating lines the curves tangent to
$E_3$ and as the generating curve a geodesic of $\s^2(c) \times
\{a\}$ or $\h^2(-c) \times \{a\}$. Each of these totally geodesic
surfaces has the Gaussian curvature zero.
\end{itemize}
\end{trm}
\section{Totally umbilical surfaces}
A {\bf principal curve} on a surface is a curve whose tangent
vectors are all contained in a principal Ricci direction and an {\bf umbilical point} on a
surface is a point  where the
principal Ricci curvatures are equal.\\
The hypersurfaces whose first and second fundamental form differ
by a constant factor are called {\bf totally
umbilical}. \\
Let $\mathcal F$ be a foliation of codimension $1$ defined on a
Riemannian manifold of dimension $3$. Let $\xi$ be a vector field
normal to the leaves of $\mathcal F$. Then the Codazzi equation
is
\begin{align}
X\langle B(Y,Z),\xi \rangle &- Y \langle B(X,Z),\xi \rangle - \langle B([X,Y],Z),\xi \rangle  \nonumber \\
&-\langle B(Y,\overline{\nabla}_X Z),\xi \rangle + \langle
B(X,\overline{\nabla}_Y Z),\xi \rangle  = \overline{R}(X,Y,\xi,Z).
\nonumber
\end{align}
If $\mathcal F$ is  totally umbilical ($B = \lambda g$) then the
first member of the equation is zero and thus
 \[
\overline{R}(X,Y,\xi,Z) = 0.
\]

This relation implies that the integral curves of $\xi$ form one of
the principal congruences of the considered space. So we have
\begin{trm}\label{tr.omb1}
If in a Riemannian manifold of dimension $3$ there is a totally
umbilical foliation $\mathcal F$ of  codimension $1$, then this is
orthogonal to a  principal congruence.
\end{trm}
 Let $(M,\overline{g})$ be a $3$-dimensional Riemannian manifold
 with isometry group of  dimension $4$. Then we have:
\begin{trm}\label{omb2}
If $(N,g) \subset (M,\overline{g})$ is a totally umbilical surface
isometrically immersed in $M$, then $N$ is totally geodesic and
$M$ is a product manifold.
\end{trm}
\begin{proof}
Taking into consideration Theorem \ref{tr.omb1}, the totally umbilical
surfaces $(N,g) \subset (M,\overline{g})$, if there are any, must
be orthogonal to one of the
 principal congruences. From Theorem \ref{trm:5} we have that the
 congruence of the systatic geodesics of $M$ admits orthogonal
 surfaces only if $M$ is $\s^2(c) \times \R$ or  $\h^2(-c) \times \R$
 and we have seen that these surfaces are totally geodesic.\\
 Hence each totally umbilical surface
 $(N,g)
\subset (M,\overline{g})$, must contain the
 systatic geodesics passing through its points. If $B = \lambda g$, with respect to  a basis
$\{X,Y\}$ of orthonormal vector fields on $N$, the Gauss equation becomes

\[
R(X,Y,X,Y) -  \overline{R}(X,Y,X,Y) =  \lambda^2.
                              \]
It follows that the wanted surface must satisfy
\[
R(X,Y,X,Y) -  \overline{R}(X,Y,X,Y)  \geq 0.
\]
We have seen (\cite{[Rm]}, \cite{[P]}) that any surface that
contains the
 systatic geodesics has the
Gaussian curvature $G = R(X,Y,X,Y)$ equal to zero. Therefore the sectional
curvature $\overline{R}(X,Y,X,Y)$ must also be zero and thus
\[
\lambda = 0.
\]
Hence we have that the second fundamental form is identically zero
and thus  $N$  is totally geodesic. We find then that $M$ is a
product space $\s^2(c) \times \R$ or $\h^2(-c) \times \R$ and that
the totally umbilical surfaces $N$ are totally geodesic.
\end{proof}

\section{Geodesics for the rotational surfaces}

The metrics (\ref{eq:Cartan-Vranceanu}) are invariant with respect
to the rotations around the $z$- axis and this leads to the study
of the rotational surfaces, of the form
\[
 X(u,v)=\left( f(u)\cos v , f(u)\sin v ,g(u) \right) ,
\]
where $0\leq v<2\pi$ and $f,g$ are real functions with $f>0$.
 In
order to obtain the geodesics of the rotational surfaces we use
the Euler-Lagrange equations
\[
\begin{cases}
 \frac{d}{dt}\Big[2\Big(\frac{f'(u)^{2}}{[1+mf(u)^{2}]^{2}}+g'(u)^{2}\Big)\dot{u}-\frac{\ell f(u)^{2}g'(u)}{1+mf(u)^{2}}\dot{v}\Big] =
 E'(u)\dot{u}^{2}+2F'(u)\dot{u}\dot{v}+G'(u)\dot{v}^{2}
 \\\\
\frac{d}{dt}\Big[2\frac{4f(u)^{2}+\ell^{2}
f(u)^{4}}{4[1+mf(u)^{2}]^{2}}\dot{v}-\frac{\ell
f(u)^{2}g'(u)}{1+mf(u)^{2}}\dot{u}\Big] =0 .
\end{cases}
\]
 We obtain
that (see \cite{[PP]}): \\
 {\bf I. The parallels} $u=u_{0}$ will be geodesics if
\begin{equation}
\frac{f'(u_0)[2+\ell^{2}f(u_0)^{2}-2mf(u_0)^{2}]}{[1+mf(u_0)^{2}]^{3}}=0,
\end{equation}
and thus we have:
\begin{enumerate}
\item[\bf $\bullet$]
the only parallels which are geodesics of the rotational surfaces
for  $\h_3$,  $\widetilde{SL}(2,\R)$,
 $\h^2\times \R$  are, just as in the Euclidean case,  those generated by the rotation of a point of the generating curve
 where the tangent is parallel to the axis of rotation ($f'=0$).
 [For these spaces we have
 $\ell^{2}\geq 2m$.]
\item[\bf $\bullet$]
 For the rotational surfaces of the product manifold $\s^2 \times \R$  the parallels which are geodesics
 have $f'=0$ or $f(u_{0})=\displaystyle \frac{\sqrt{m}}{m}.$
In this case we have $\ell^{2}<2m$.
\item[\bf $\bullet$]
For the rotational surfaces of
$SU(2)$, besides the  parallels with $f'=0$, there are the
parallels for which $f(u_{0})=\sqrt{\frac{2}{2m-l^{2}}}$.
\end{enumerate}
\noindent {\bf II. The meridians}  $v=v_{0}$  are geodesics if
 $$\frac{\ell f(u)^{2}\sqrt{[1+mf(u)^{2}]^{2} -f'(u)^{2}}}{{[1+mf(u)^{2}]^{2}}}=const.$$
It follows that
\begin{enumerate}
\item[\bf $\bullet$]
 if $\ell = 0$ then for the rotational surfaces of the product manifold  $\s^2\times \R$ and $\h^2 \times \R$, as in the euclidian case, all the meridians are geodesics;
\item[\bf $\bullet$]
all the meridians of the cylinders $f(u) = const.$ are geodesics;
\item[\bf $\bullet$]
 if  $\ell  \neq 0$  the  meridians are geodesics for $m \gtreqless 0$ if the function $f$ is
 \[
 f(u) = \frac{\tan(\sqrt{m} u + c)}{\sqrt{m}}, \quad
 \qquad
 f(u) = u , \qquad
 f(u) = \frac{\tanh(\sqrt{-m} u + c)}{\sqrt{-m}}
 \]
 or if $f$ is a solution of the equation
\begin{align}
2 f'(u)& + 4 m f(u)^2 f'(u)+
2 m^2 f(u)^4 f'(u)-2 f'(u)^3 \nonumber \\
&+ 2 m f(u)^2 f'(u)^3 - f(u) f'(u) f''(u)-m f(u)^3 f'(u) f''(u) =
0. \nonumber
\end{align}
\end{enumerate}
 In the particular case of the {\bf cylinder} of equation
\begin{equation}
S(u,v)=\big(a\cos v,a\sin v,u\big),\quad a\in\mathbb{R},
\end{equation}
we obtain the following
\begin{prop}
The geodesics of the cylinder are the curves of equation
$$
\gamma(s)=(a\cos(As+B), a\sin(As+B), Cs+D).
$$
that include:
 \begin{itemize}
\item[\bf $\bullet$]  the meridians;
\item[\bf $\bullet$]   the parallels,
\item[\bf $\bullet$]  the {\it  helices}, that is curves with constant
geodesic curvature and geodesic torsion, analogous of the helices
of $\R^{3}$.
\end{itemize}
\end{prop}

\strut\hfill Universit\`a degli Studi di Cagliari,\\
\strut\hfill Dipartimento di Matematica e Informatica\,\\
\strut\hfill Via Ospedale 72, 09124 Cagliari, ITALIA\\

\strut\hfill piu@unica.it\\
\strut\hfill profirmanuela@yahoo.com
\bigskip

\end{document}